\documentclass[11pt,a4paper]{article}
\usepackage[utf8]{inputenc} 
\usepackage[T1]{fontenc}    
\usepackage{lmodern}        
\usepackage[english]{babel} 
\usepackage{amsmath,amssymb,amsfonts} 
\usepackage{graphicx}       
\usepackage{hyperref}       
\usepackage{geometry}       
\usepackage[normalem]{ulem} 
\usepackage{ifthen}
\newlength{\wdTempA}
\newlength{\wdTempB}

\usepackage{bussproofs}
\usepackage{proof}
\usepackage{url}
\usepackage{mathrsfs} 
\usepackage{enumerate} 
\usepackage{paralist} 
\usepackage{amsthm}
\usepackage{amsmath,amsfonts}
\usepackage{bm}
\usepackage{amssymb} 
\usepackage{comment}
\usepackage{ulem}
\usepackage{ascmac}

\usepackage{graphicx}
\usepackage{xcolor}
\usepackage{hyperref}
\hypersetup{
    setpagesize=false,
    bookmarksnumbered=true,
    bookmarksopen=true,
    colorlinks=true,
    linkcolor=blue,
    citecolor=red,
}

\usepackage{tikz} 
\usetikzlibrary{cd} 
\usepackage{here}
\usepackage{braket} 

\numberwithin{equation}{section}
\theoremstyle{definition}
\newtheorem{theorem}{Theorem}[section] 
\newtheorem*{theorem*}{Theorem}
\newtheorem{definition}[theorem]{Definition}
\newtheorem*{definition*}{Definition}
\newtheorem{lemma}[theorem]{Lemma}
\newtheorem{proposition}[theorem]{Proposition}
\newtheorem{remark}[theorem]{Remark}

\newtheorem{question}[theorem]{Question}

\newtheorem{corollary}[theorem]{Corollary}

\newcounter{claimcounter}[theorem]
\newtheorem{claim}[claimcounter]{Claim}
\newenvironment{proofofclaim}{\noindent proof of claim.}{\hfill.}

\newcommand\Con{\mathrm{Con}}

\newcommand\N{\mathbb{N}}

\newcommand\Q{\mathbb{Q}}

\newcommand\Lt{\mathrm{L}_2}

\newcommand\ko{\mathcal{O}}

\newcommand\Rcaz{\mathsf{RCA}_0}
\newcommand\Wklz{\mathsf{WKL}_0}
\newcommand\Acaz{\mathsf{ACA}_0}
\newcommand\Rfn{\mathsf{RFN}}
\newcommand\Rfnz{\mathsf{RFN}_0}

\newcommand\JI{\mathsf{JI}}
\newcommand\Acz{\mathsf{AC}_0}
\newcommand\Caz{\mathsf{CA}_0}
\newcommand\Dcz{\mathsf{DC}_0}
\newcommand\Tdcz{\mathsf{TDC}_0}

\newcommand\HYP{\mathrm{HYP}}

\newcommand\Atrz{\mathsf{ATR}_0}
\newcommand\ATR{\mathsf{ATR}}
\newcommand\Trz{\mathsf{TR}_0}
\newcommand\weak{\mathrm{weak}~}
\newcommand\unique{\mathrm{unique}~}
\newcommand\finite{\mathrm{finite}~}
\newcommand\nfm{\mathrm{~nonzero~finitely ~many~}}

\newcommand\nsfm{\mathrm{~nonzero~subfinitely ~many~}}

\newcommand\bounded{\mathrm{bounded}~}
\newcommand\Ind{\mathrm{IND}}

\newcommand\field{\mathrm{field}}

\newcommand\codedom{\mathrm{coded}~\omega\bou\mathrm{model}~}

\newcommand\TJ{\mathrm{TJ}}

\let\oldbigcup\bigcup
\renewcommand\bigcup{\textstyle\oldbigcup}
\let\oldbigcap\bigcap
\renewcommand\bigcap{\textstyle\oldbigcap}
\let\oldprod\prod
\renewcommand\prod{\textstyle\oldprod}
\let\oldint\int
\renewcommand\int{\displaystyle\oldint}
\renewcommand\subset{\subseteq}
\newcommand\bou{\mathchar`-}

\newcommand{\adunderbrace}[2]{%
    \settowidth{\wdTempA}{$#1$}%
    \settowidth{\wdTempB}{${\scriptstyle #2}$}%
    \ifthenelse{\wdTempA<\wdTempB}{%
        \hspace*{.5\wdTempA}\hspace*{-.5\wdTempB}%
        \underbrace{#1}_{#2}%
        \hspace*{.5\wdTempA}\hspace*{-.5\wdTempB}%
    }{%
        \underbrace{#1}_{#2}%
    }%
}

\title{Approximation of hyperarithmetic analysis\\ by $\omega$-model reflection}
\author{Koki Hashimoto}
\date{\today}

\begin{document}

\maketitle
\begin{abstract}
This paper presents two types of results related to hyperarithmetic analysis.
First, we introduce new variants of the dependent choice axiom, namely
$\mathrm{unique}~\Pi^1_0(\mathrm{resp.}~\Sigma^1_1)\bou\mathsf{DC}_0$ and 
$\mathrm{finite}~\Pi^1_0(\mathrm{resp.}~\Sigma^1_1)\bou\mathsf{DC}_0$.
These variants imply $\mathsf{ACA}_0^+$ but do not imply $\Sigma^1_1\mathrm{~Induction}$.
We also demonstrate that these variants belong to hyperarithmetic analysis and explore their implications with well-known theories in hyperarithmetic analysis.
Second, we show that $\mathsf{RFN}^{-1}(\mathsf{ATR}_0)$, a class of theories defined using the 
$\omega$-$\mathrm{model~reflection~axiom}$, approximates to some extent hyperarithmetic analysis, and investigate the similarities between this class and hyperarithmetic analysis.

\end{abstract}

\tableofcontents 
\newpage

\section{Introduction}
The research program called Reverse Mathematics was initiated by Harvey Friedman in the mid-1970s and further developed by Stephen Simpson in the 1980s.
This program classifies mathematical theorems according to the strength of set existence axioms required for their proofs.In fact, the majority of classical mathematical theorems can either be proven within a subsystem of second-order arithmetic ($\mathsf{Z}_2$), denoted as $\Rcaz$, or shown to be equivalent to one of $\Wklz$, $\Acaz$, $\Atrz$, or $\Pi^1_1\bou\Caz$ over $\Rcaz$ \footnote{When demonstrating this equivalence, proofs proceed not only from axioms to theorems but also from theorems to axioms, hence the term ``reverse mathematics'' referring to this contrary process to conventional mathematics}.
These five theories, referred to as the ``big five'', have become the most extensively researched subsystems of second-order arithmetic to date.

From the early stages of reverse mathematical research, many theorems were proven to be equivalent to one of the big five axioms, but simultaneously, numerous theorems were also discovered that did not correspond to any of the big five.As research progressed, these theorems, which were not equivalent to any of the  big five, increasingly drew the attention of researchers.

The hyperarithmetic analysis, which is the focus of this paper, belongs to the category of theories that do not belong to any of the big five.

The term ``hyperarithmetic analysis'' has been used since the 1960s with a slightly different meaning than the definition given by Antonio Montalb{\'{a}}n \cite{Montalbn2006IndecomposableLO}  today.
\begin{definition*}[]\label{definition 超算術的解析}
   A set of $\Lt$ sentences $T$ is \textit{a theory of  hyperarithmetic analysis} if it satisfies the following conditions.\footnote{
		This condition is equivalent to saying that
every $X\subset \omega$, $\HYP(X)$ is the minimum $\omega$-model of $T+\Rcaz$ which contains $X$.}.
	\begin{itemize}
		\item every $X\subset \omega$, the $\omega$ model $\HYP(X)$ consisting of all $X$-hyperarithmetical sets satisfies $T$.
		\item if an $\omega$-model $M$ satisfies $T$, then $M$ is closed under computable union and hyperarithmetical reduction.
\end{itemize}
\end{definition*}
To the best of the author's knowledge, Friedman's 1967 doctoral dissertation \cite{Friedman1967SubsystemsOS} represents the earliest occurrence of the term ``hyperarithmetic analysis'' within the context of second-order arithmetic literature.
In that doctoral dissertation, the term ``hyperarithmetic analysis'' was used to refer to the theories which have HYP as their minimum $\omega$-model.
Even before the onset of reverse mathematical research, the fact that theories of hyperarithmetic analysis characterize the class $\HYP$ was well-known, providing significant motivation for studying these theories.

In the early stages of reverse mathematical research, primarily centered around the 1970s, two theorems belonging to hyperarithmetic analysis were mentioned: theorem $\mathsf{SL}$ related to the convergence of sequences and the arithmetic variant of the Bolzano–Weierstrass theorem $\mathsf{ABW}$ 	\cite{Friedman1975}.However, while these are theorems in analysis, their formulation involved references to arithmetic formulas, making them not purely mathematical theorems.

Even as reverse mathematical research became more active in the 1980s, no mathematical theorems belonging to hyperarithmetic analysis were found besides those two.It can be said that research on hyperarithmetic analysis, unable to ride the wave of reverse mathematics, stagnated for some time.

An event that significantly changed this situation was the discovery in 2006 by Montalbán of pure mathematical theorems belonging to hyperarithmetic analysis, which he called $\mathsf{INDEC}$ \cite{Montalbn2006IndecomposableLO}.
Furthermore, in the same paper, he not only introduced five other theories belonging to hyperarithmetic analysis besides $\mathsf{INDEC}$ but also salvaged the construction method of $\omega$-models called tagged tree forcing developed by Steel \cite{STEEL197855}, and further arranged it to facilitate the separation of theories in second-order arithmetic.Since then, research on separating theories of hyperarithmetic analysis using tagged tree forcing has become active.
For example, as a recent related result, Jun Le Goh proved the incomparability of a variant of the finite choice axiom $\finite\Pi^1_0\bou\Acz$\footnote{
	In most literature, it is denoted as ``$\finite\Sigma^1_1\bou\Acz$''.I believe that this convention should be discontinued (cf.remark \ref{remark:Yamero_Sigma11_Pi10}).
} with $\Delta^1_1\bou\Caz$ using tagged tree forcing \cite{goh_2023}.
In addition to research using tagged tree forcing methods, the 2020s saw the discovery of another pure mathematical theorem belonging to hyperarithmetic analysis, following $\mathsf{INDEC}$.This theorem, known as Halin's theorem in graph theory, is detailed in \cite{Barnes2023HalinsIR}.
These findings show that hyperarithmetic analysis, although initially delayed, is now a key component of Reverse Mathematics.

The results reported in this paper regarding hyperarithmetic analysis can be classified into two types.
The first involves defining $\unique\Pi^1_0\bou\mathsf{DC}_0$ and its variants, and examining their relative strengths compared to existing theories of hyperarithmetic analysis.
The $\unique\Pi^1_0\bou\mathsf{DC}_0$ is easily seen to be strictly stronger than $\Acaz^+$ and strictly weaker than $\Sigma^1_1\bou\mathsf{DC}_0$.
Furthermore, in this paper, we show that while this theory is incomparable with $\Sigma^1_1\bou\mathsf{AC}_0$ and $\Delta^1_1\bou\mathsf{CA}_0$, it is equivalent under appropriate induction to $\unique\Pi^1_0\bou\Acz$(conventionally denoted as $\weak\Sigma^1_1\bou\Acz$).
Using this result as a stepping stone, we introduce several variants of $\unique\Pi^1_0\bou\mathsf{DC}_0$ and discuss their strength relative to existing theories of hyperarithmetic analysis.Figure \ref{figure:Reverse Mathematics Zoo: Hyperarithmetic Analysis Area} organizes the obtained results.

The second effort involves attempts to syntactically characterize hyperarithmetic analysis, the class traditionally defined through $\omega$-models.To the best of the author's knowledge, such a characterization has not yet been established.Unfortunately, even the author has not achieved a complete syntactical description; however, it was confirmed that the class $\Rfn^{-1}(\Atrz)$, defined below, provides some approximation.A more detailed explanation follows.

Let $T$ be a $\Lt$ theory.We call the statement ``Given $X \subset \N$, there exists a coded $\omega$-model $M$ such that $M$ contains $X$ and satisfies $T+\Acaz$ '' \textit{$\omega$-model reflection} of $T$ and denote it as $\Rfn(T)$.
We denote the set of all \( T \) for which \( \Rfn(T) \) is equivalent to \( \Atrz \) over \( \Acaz \) as \( \Rfn^{-1}(\Atrz) \).First, we confirm that this approximates the class of hyperarithmetic analysis to some extent.
Specifically, we demonstrate the similarities between hyperarithmetic analysis and $\Rfn^{-1}(\Atrz)$ in terms of individual instances, and show that they share certain closure properties.
We also consider two types of questions regarding the relationship between $\Rfn^{-1}(\Atrz)$ and hyperarithmetic analysis.

These questions seek to explore what other similarities exist, and conversely, what differences there are between them.
In this paper, as an example of the former, we discuss whether $\Sigma^1_3$ instances exist in $\Rfn^{-1}(\Atrz)$.This question arises from the fact that $\Sigma^1_3$ instances are known not to exist in hyperarithmetic analysis.
Regrettably, the author has not completely resolved the question.Nonetheless, the paper offers a partial solution by showing that the assertion is valid for $\forall X \exists ! Y \Pi^1_0$ type of $\Pi^1_2$-sentences.

On the other hand, as a question in the direction of the latter, we can consider what theories exist in symmetric difference of $\Rfn^{-1}(\Atrz)$ and hyperarithmetic analysis.However, only trivial theories have been found so far, and fundamentally, this remains unresolved.

In Chapter 2, we organize the foundational topics of second-order arithmetic with a focus on hyperarithmetic analysis and introduce related prior research.
In Chapter 3, we introduce the unique version of the dependent choice principle, denoted as $\unique\Gamma\bou\Dcz$, and explore the implications between this axiom and various other axioms, including variations of the axiom of choice.
In Chapter 4, we introduce the finite version of the dependent choice principle, denoted as $\finite\Gamma\bou\Dcz$, and conduct an analysis similar to that in Chapter 3.
In Chapter 5, we verify that all theories with strengths between $\JI_0$ and $\Sigma^1_1\bou\Dcz$ belong to $\Rfn^{-1}(\Atrz)$, demonstrating the closure properties of this class, such as closure under the addition of induction axioms.We then show that there are no statements of the form $\forall X \exists ! Y \theta(X,Y)$ (where $\theta$ is an arithmetic formula) that belong to $\Rfn^{-1}(\Atrz)$.

\newpage

\begin{figure}[hbtp]
	\centering
	\includegraphics[width=14cm,pagebox=cropbox,clip]{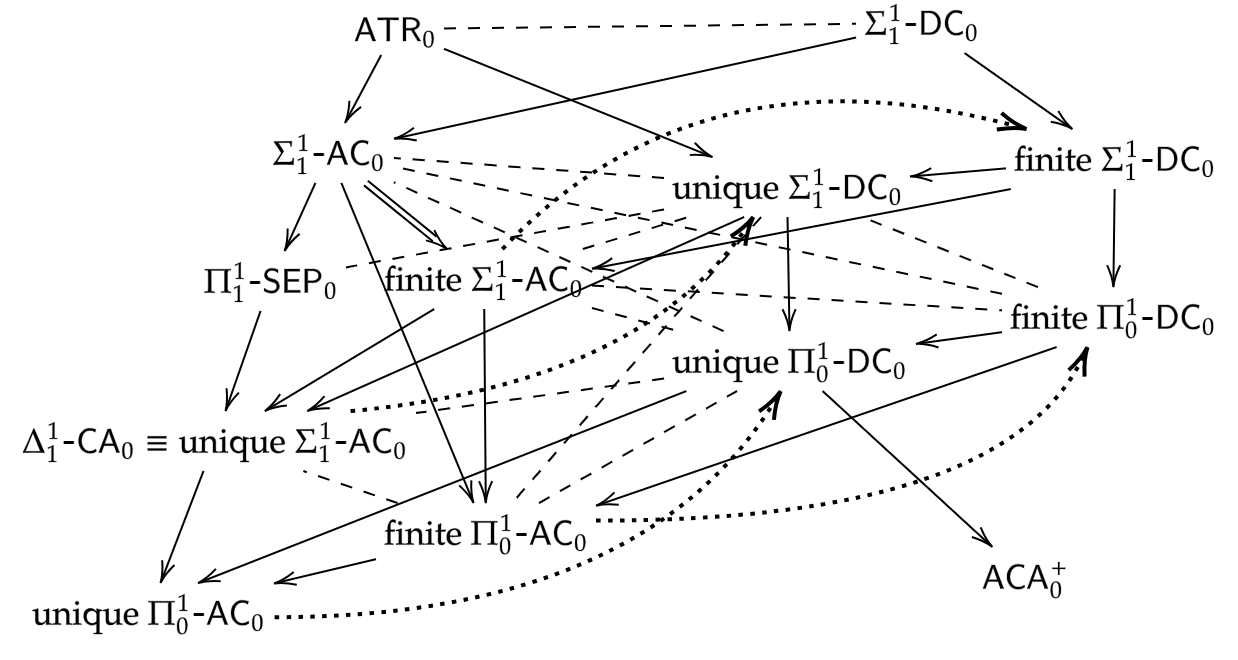}
	\caption{The obtained Reverse Mathematics Zoo: Hyperarithmetic Analysis Area}\label{figure:Reverse Mathematics Zoo: Hyperarithmetic Analysis Area}
\end{figure}

In the picture, the fact that there is a double arrow extending from $T$ to $S$ means that $S$ is provable from $T$.In the case of a single arrow, in addition to the above, it means that the reverse is unprovable.Also, the fact that they are connected by a dashed line means that either one cannot prove the other.Finally, when there is a curved dotted arrow extending from $T$ to $S$, it means that $S$ is provable from $T + (\Sigma^1_{k+1}~\mathrm{Induction})$ by appropriately choosing $k\in\omega$.

\section*{Acknowledgment}
The implication diagrams used in this paper were created using the mathematical editor Mathcha (\url{https://www.mathcha.io/editor}), introduced to me by Kaito Takayanagi.It not only facilitated organizing the results but also led to the discovery of new implications.

Valuable feedback was received from Dr.Yudai Suzuki, who pointed out an error in my initial proof of Lemma \ref{lemmma Ref(T)⇔Ref(T+￢Ref(T))}.Moreover, the formulation of this lemma was prompted during a discussion where he suggested that the lemma could potentially address a more sophisticated claim than I had initially considered.

Throughout the research, extensive guidance was received from my supervisor, Takayuki Kihara, on topics such as computability theory and the fundamentals of second-order arithmetic.Many of the results regarding $\unique\Gamma\bou\Dcz$ and $\finite\Gamma\bou\Dcz$ in this study were obtained through discussions with him.

Including the aforementioned three individuals, I would like to express my gratitude to everyone who engaged in discussions at various research meetings.

\section{Preliminaries}
The language of second-order arithmetic, $\Lt$, is a two-sorted first-order language.It is derived from the language of first-order arithmetic by adding set variables such as $X, Y, Z, \ldots$ and the membership predicate $\in$.

We introduce the notation used in this paper, along with the basic theory of second-order arithmetic.Initially, $\langle n, m \rangle$ denotes the pairing of $n$ and $m$.
\begin{align*}
 X\oplus Y &:= \set{2n | n\in X } \cup \set{2n+1 | n\in Y }.\\
 X = Y &:\leftrightarrow \forall n ( n\in X \leftrightarrow n\in Y).\\
 Y_k &:= \set{n | \braket{n,k}\in Y}.\\
 X\in Y &:\leftrightarrow \exists k\forall n( n\in X \leftrightarrow Y_k).\\
 X\times Y &:= \set{\braket{x,y} | x\in X ,y\in Y}.\\
 Y^k &:= \bigcup_{l<k} Y_l\times\{l\} = \set{\braket{x,l} | x\in Y_l ,l < k}.\\
 \Phi^X_e(x) &:= \text{The result of inputting } x \text{ into the } e\text{-th Turing machine using } X \text{ as the oracle.}
\end{align*}
When $\Phi^X_e$ is a function from $\N$ to $\{0,1\}$, we write $\Phi^X_e \in 2^\N$.

$M \subset \N$ can be viewed as the code for the family of sets $\braket{M_k}_{k \in \N}$.Then, the set $M$ can be considered as the code of $\omega$-model $\braket{\N, \braket{M_k}_{k \in \N}, +, \cdot, <, 0, 1}$.When viewed in this way, $M$ is referred to as \textit{coded $\omega$-model}.
Over $\Acaz$, for each $\Lt$ formula $\varphi$, the truth of $\varphi$ in model $M$ can be determined by taking a number of Turing Jumps corresponding to its complexity.That is, the existence of a function $f: \mathrm{Subst}^M(\varphi) \to \{\mathrm{True}, \mathrm{False}\}$ can be proven (\cite{simpson_2009} Lemma VII.2.2).Here, $\mathrm{Subst}^M$ represents the entire set of substitution instances for $\varphi$.Moreover, from the strong soundness theorem, if a $\codedom$ satisfying $\sigma$ can be constructed, then the consistency of $\sigma$ can be derived over $\Acaz$(\cite{simpson_2009} Theorem II.8.10).


The induction axiom scheme for the class of formulas $\Gamma$, denoted $\Gamma\bou\mathsf{IND}$, is as follows:
\[ \varphi(0) \land \forall n (\varphi(n) \rightarrow \varphi(n+1)) \rightarrow \forall n \varphi(n) \quad \mathrm{where}~ \varphi \in \Gamma.\]
Induction reduces the complexity of formulas.For example, in first-order arithmetic, arithmetic induction (more precisely, the collection axiom scheme) ensures that the class of arithmetic formulas are closed under bounded quantification of numbers.Similarly, over $\Rcaz + \Sigma^1_{k}\bou\Ind$, $\Sigma^1_k$ formulas are closed under bounded universal quantification of numbers.Formally, the following lemma holds.

\begin{lemma}\label{lemma:Σ1k上Σ1kは有界任意量化に閉じる}
 Let $k \in \omega$.The following hold in $\Rcaz + \Sigma^1_{k+1}\bou\Ind$:
	\[\forall n <m \exists  Y \varphi(n,Y) \rightarrow \exists Z \forall n <m \varphi(n,Z_n) ~~~ \mathrm{where}~\varphi\in \Pi^1_{k}.\]
\end{lemma}

\subsection{Tree notation}
In $\Rcaz$, fix a coding of finite sequences by primitive recursive functions, and denote the set of all finite sequences as $\N^{<\N} (\subset \N)$.A set $T \subset \N^{<\N}$ is called a tree if it is closed under initial segments.For $\sigma \in T$, let $|\sigma|$ denote the length of the finite sequence $\sigma$ (e.g., $|\langle 0,3,1,6 \rangle| = 4$).Operations such as the concatenation of finite sequences, $\sigma^\frown\tau$, are naturally defined.In particular, to explicitly denote that $\sigma$ is an element of a binary tree of length $l$, the notation $\sigma \in 2^l$ is used.A subset $X \subset \N$ is a path through $T$ if $\forall n (X[n] \in T)$ holds, where $X[n]$ is the finite sequence $\langle X(0), \ldots, X(n-1) \rangle$.Of course, $X(i) = \begin{cases} 1 & \text{if } i \in X \\ 0 & \text{otherwise} \end{cases}$.Following this notation of $X[i]$ and $X(i)$, $\sigma[i], \sigma(i)$ for a finite sequence $\sigma$ are similarly defined.




\subsection{Formalized hyperarithmetical sets}
In this paper, we define hyperarithmetical sets within formal systems through the coding of computable well-orderings.For details, refer to \cite{simpson_2009}, Chapter VIII.

\begin{definition}
We denote by $a \in \ko^X_+$ the following $\Lt$ formula that has only $a$ and $X$ as free variables:
\[\exists e,i\le a [ \braket{e,i} = a \land \Phi^X_e \in (2^\N\cap \mathrm{LO}) \land \Phi^X_e(\braket{i,i}) = 1 ].\]
In the above, $Y \in \mathrm{LO}$ means that $Y$ is the code of a linear ordering.
\end{definition}

\begin{definition}
The left-hand side is abbreviated as the right-hand side as follows:
\begin{itemize}
  \item  $i \le^X_e j :\leftrightarrow  \braket{e,j} \in \ko^X_+ \land \Phi^X_e(\braket{i,j})=1 $.
  \item $a \le^X_\ko b :\leftrightarrow \exists e,i,j( \braket{e,i} = a \land \braket{e,j} = b \land i\le^X_e j)$.
  \item $a <^X_\ko b :\leftrightarrow a \le^X_\ko b \land a\neq b$.
  \item $a \in \ko^X :\leftrightarrow a\in \ko^X_+ \land  \lnot \exists f\in\N^\N\forall n (f(n+1)<^X_\ko f(n) <^X_\ko a)$.
\end{itemize}
\end{definition}

\begin{remark}\label{remark:クリーねのOに関する論理式の複雑さ}
The statement $a \in \ko^X_{+}$ is $\Pi^0_2$  and $a \in \ko^X$ is $\Pi^1_1$ over $\Acaz$.Furthermore, when $a \in \ko^X_+$, the set $\{\, b \mid b \le^X_\ko a \,\}$ is $\Sigma^0_1$ definable.
\end{remark}


\begin{definition}[H set]
For $a$ such that $a \in \ko^X$, we denote the set obtained by starting with $X$ and repeatedly applying the Turing jump $a$ times (if possible) as $\mathrm{H}^X_a$\footnote{Several natural definitions can be considered.For example, $\mathrm{H}^X_a = \bigcup_{b <^X_\ko a} (X \oplus \TJ(\mathrm{H}^X_b)) \times \{b\}$ might be one of them.}.Furthermore, the assertion that such $\mathrm{H}^X_a$ always exists if $X$ is a well-order is equivalent to $\ATR$ over $\Rcaz$.Therefore, in this paper, we define $\Atrz$ as this assertion $+\Rcaz$.

\end{definition}


\subsection{Theories of hyperarithmetic analysis}




The top four axioms schemes below have been extensively studied in the context of hyperarithmetic analysis since its earliest stages around the 1970s.
In contrast, the bottom-most axiom, $\JI$, is relatively new and was introduced in \cite{Montalbn2006IndecomposableLO} \footnote{
Although the original and our $\JI$ appear different, they are equivalent.
}.
\begin{align*}
	\Sigma^1_1\bou\mathrm{DC}~:~&\forall X \exists  Y \varphi(X,Y) \rightarrow  \exists Y  \forall n \varphi(Y^n ,Y_{n}) ~~~ \mathrm{where}~\varphi\in \Sigma^1_1.\\
	\Sigma^1_1\bou\mathrm{AC}~:~& \forall n  \exists  Y \varphi(n,Y) \rightarrow \exists Z \forall n \varphi(n,Z_n) ~~~ \mathrm{where}~\varphi\in \Sigma^1_1.\\
	\Delta^1_1\bou\mathrm{CA}~:~& \forall n (\varphi(n)\leftrightarrow \psi(n)) \rightarrow  \exists \set{n|\varphi(n)} ~~~ \mathrm{where}~\varphi\in \Sigma^1_1,\psi\in\Pi^1_1.\\
	\unique\Pi^1_0\bou\mathrm{AC}~:~& \forall n  \exists ! Y \varphi(n,Y) \rightarrow \exists Z \forall n \varphi(n,Z_n) ~~~ \mathrm{where}~\varphi\in \Pi^1_0.\\ 
	\JI~:~&\forall X[\forall a\in\ko^X((\forall b <_{\ko^X}a	\exists\mathrm{H}^X_b) \rightarrow \exists\mathrm{H}^X_a)].
	\end{align*}
As usual, we denote the theories obtained by adding $\Rcaz$ to these by appending $0$ below each.
$\mathsf{AC}$ stands for the Axiom of Choice, and $\JI$ for Jump Iteration.

These theories become strictly weaker from top to bottom.
Even in the weakest theory $\JI_0$ among the above, $\Acaz$ is implied, thus all the theories listed above imply $\Acaz$.
Simply proving the implication from the theory above to the theory below is easy.On the other hand, it is often necessary to construct an $\omega$-model to separate these theories, which is generally difficult \footnote{
	An $\omega$-model for separating $\Sigma^1_1\bou\mathrm{DC}_0$ and $\Sigma^1_1\bou\mathrm{AC}_0$ can relatively easily be proven to exist using a well known theorem : $\omega$-model incompleteness (\cite{simpson_2009} Theorem VIII.5.6).For separating other theories, all $\omega$-models are constructed using the technique called tagged tree forcing.
	}.
\begin{remark}\label{remark:Yamero_Sigma11_Pi10}
$\unique \Pi^1_0\bou\Acz$ has been referred to as $\weak\Sigma^1_1\bou\Acz$ in various literature so far.
However, in practice, it is easily proven that $\unique \Sigma^1_1\bou\Acz$ is equivalent to $\Delta^1_1\bou\Caz$, and this is stronger than $\unique \Pi^1_0\bou\Acz$.
Therefore, despite the essential difference between $\unique \Pi^1_0\bou\Acz$ and $\unique \Sigma^1_1\bou\Acz$, many researchers have probably followed the notation using \textbf{weak} that confuses the two, with some resistance, because it is customary.
As no particular benefits are found in this conventional notation, it is about time to consider a refresh.
\end{remark}
\begin{proposition}
The following equivalences hold of $\unique\Gamma\bou\Acz$.
\begin{enumerate}
	\item $\unique\Pi^0_1\bou\Acz \equiv \Wklz$~(folklore?).
	\item $\unique\Sigma^0_2\bou\Acz \equiv \Acaz$.
	\item $\unique\Pi^0_2\bou\Acz \equiv \unique\Pi^1_0\bou\Acz $.\label{equivalence:Pi2andarith}
\end{enumerate}
\end{proposition}
\begin{proof}
1.It suffices to show that $\unique\Pi^0_1\bou\Acz \vdash \Wklz$.
Assume $\unique\Pi^0_1\bou\Acz \not\vdash \Wklz$.For the sake of contradiction, We will show $\unique\Pi^0_1\bou\Acz + \lnot \Wklz \vdash \Sigma^0_1\bou\Caz$.
Let $\theta(n,t)$ be $\Delta^0_0$.For simplicity, we assume that if there exists $t$ such that $\theta(n,t)$ then the witness $t$ is unique.

First, by $\lnot \Wklz$, we take a infinite binary tree $T_E$ with no path.Next, we define a sequence of recursive trees, $\{T_n\}_{n \in \N}$, each with exactly one path, as follows.
Note that at each step $t$, a finite number of vertices are added at height $t$.

We search for $t$ satisfying $\theta(n,t)$ for $t = 0, 1, 2, \ldots$.
And then, while it seems that $\exists t \theta(n,t)$ is false, that is, while $\lnot \theta(n,t)$ holds for $t$, we grow trees as $0,00,000,0000,...$.
If at some step $t$ $\theta(n,t)$ holds, then we start growing as $\sigma^\frown 0,\sigma^\frown \braket{00},\sigma^\frown \braket{000},\sigma^\frown \braket{0000},...$from the leftmost vertex $\sigma$ among the vertices at height $t$ in $T_E$.
The following is a formal representation of the above construction in $\Lt$ formula:
\begin{align*}
	\sigma \in T_n &\leftrightarrow \\
	\sigma = \braket{ }\lor &[\forall l<|\sigma|(\sigma(l)=0) \land \forall l < |\sigma| \lnot \theta(n,l) ] \lor [ \sigma(0)=1 \land \{\exists \tau ( \sigma = 1^\frown \tau \land \\
	&[([\forall l < |\tau| \lnot \theta(n,l) \land \tau \in T_E ])\lor( \exists  s <|\tau| (\theta(n,s) \land \tau[s]= \sigma_{LM(t) \mathrm{in} T_E } \\
	&~~~~~~~~~~~~~~~~~~~~~~~~~~~~~~~~\land \forall t(s \le t < |\tau|\rightarrow \tau(t)= 0  )  ) ) ] ) \} ].	
\end{align*}
Here, $\sigma_{LM(t) \mathrm{in} T_E}$ represents the leftmost vertex at height $t$ in $T_E$.

Since $T_E$ does not have any paths, but is an infinite tree nonetheless, each $T_n$ is an infinite tree with exactly one path.
Therefore, $\forall n \exists ! X \forall k (X[k]\in T_n) $, so by assumption, we can take $\braket{ X_n}_{n\in \N}$ such that $\forall n (X_n[k]\in T_n)$ holds.
Thus, $\set{n | X_n(0) = 1 } = \set{n | \exists t \theta(n,t) }$ holds.
2.Since the reverse is well known, it suffices to show that $\unique\Sigma^0_2\bou\Acz$ implies $\Acaz$.To see this, we prove $\unique\Sigma^0_2\bou\Acz \vdash \Sigma^0_1\bou\Caz$.
Let $\varphi(n)$ be $\Sigma^0_1$.Write $\varphi(n)$ as $\exists t\theta(n,t)$ where $\theta(n,t)$ is $\Delta^0_0$.
Define a $\Sigma^0_2$ formula $\psi(n,X)$ by
\[ [X= \{0\} \land \lnot \exists t \theta(n,t) ]\lor \exists t [ X= \{t+1\}\land \theta(n,t) \land \forall s<t \lnot \theta(n,s) ].\]
It is easy to check that $\forall n\exists ! X \psi(n,X)$, so we can take a set $X = \braket{X_n}_{n\in \N}$ such that $\forall n \psi(n,X_n)$ .
Then $\set{n | 0\not\in X_n } = \set{n |\exists t \theta(n,t) }$.
3.\cite{SuzukiweakAC} Theorem 6.
\end{proof}

\begin{proposition}
$\JI_0$ and $\Sigma^1_1\bou\Dcz$ are hyperarithmetic analysis.
Furthermore, since a theory whose strength lies between two theories of hyperarithmetic analysis is also hyperarithmetic analysis, theories $\Sigma^1_1\bou\mathrm{AC}_0$, $\Delta^1_1\bou\mathrm{CA}_0$ and $\unique\Pi^1_0\bou\mathrm{AC}_0$ are also hyperarithmetic analysis.
\end{proposition}
\begin{proof}
Since $\Sigma^1_1\bou\Dcz\vdash\JI$ holds, it is sufficient to verify the following two.
\begin{enumerate}
  \item every $X\subset \omega$, $\HYP(X)\models \Sigma^1_1\bou\Dcz$.
\item every $\omega$-model of $\JI$ is closed under computable union and hyperarithmetic reducibility.
\end{enumerate}
	$1.$  \cite{simpson_2009} Theorem VIII.4.11.
	$2.$ Let $M$ be an $\omega$-model of $\JI_0$, and let $X \in M$ be fixed.
	Since \( M \) is specifically a model of \( \mathcal{A}_0 \), it is closed under Turing reducibility.
	Therefore, it is sufficient to show that for all $a \in \ko^X$, $\mathrm{H}^X_a \in M$ holds.
	Here, it should be noted that $\ko^X$ is not on $M$, but rather at the meta-level.

We show the previous assertion by transfinite induction on $<_\ko^X$ at the meta-level.
First, by the absoluteness of arithmetic formulas over the $\omega$-model, since for each $a \in \mathcal{O}^X$, $M \models a \in \mathcal{O}^X$ holds, it follows from Remark~\ref{remark:クリーねのOに関する論理式の複雑さ} that
\[
\{ b \in \omega \mid b <_{\mathcal{O}^X} a \} = \{ b \in \omega \mid M \models b <_{\mathcal{O}^X} a \} \in M.
\]
Therefore, by the inductive hypothesis $\forall b<_{\ko^X} a( \mathrm{H}^X_b \in M)$, $M\models \forall b<_{\ko^X} a\exists  \mathrm{H}^X_b $ is implied, and from $M\models \JI$, $\mathrm{H}^X_a \in M$ can be obtained.
\end{proof}
\begin{remark}
To the best of the author's knowledge, all hyperarithmetic theories defined axiomatically lie between $\JI_0$ and $\Sigma^1_1\bou\Dcz$.
\end{remark}

The five theories introduced above are each separated by the $\omega$-model.Among them, there are several $\omega$-models that will be used in the following sections.Therefore, along with introducing these models, let's also introduce previous research.

\begin{theorem}[\cite{VanWesep1977} Theorem 1.1]\label{theorem:M_w}
	There exists an $\omega$-model, denoted as $M_w$, satisfying $\unique\Pi^1_0\bou\Acz$ and not satisfying $\Delta^1_1\bou\Caz$.
\end{theorem}
\begin{theorem}[\cite{Friedman1967SubsystemsOS} II.4 Theorem 2]
There exists an $\omega$-model satisfies $\Sigma^1_1\bou\Acz$ and does not satisfies $\Sigma^1_1\bou\Dcz$.
\end{theorem}

\begin{theorem}[\cite{STEEL197855} Theorem 4]
There exists an $\omega$-model satisfies $\Delta^1_1\bou\Caz$ and does not satisfies $ \Sigma^1_1\bou\Acz$.
\end{theorem}
As seen below, $\Delta^1_1\bou\Caz$ and $ \Sigma^1_1\bou\Acz$ are separated more finely.
The theory obtained by adding $\Acaz$ to the following diagram is called $\Pi^1_1\bou\mathsf{SEP}_0$.
	\[\lnot \exists n[\psi_0(n) \land \psi_1(n)]\rightarrow \exists X \forall n(  (\psi_0(n) \rightarrow n\in X)\land (\psi_1(n) \rightarrow n\not\in X) ) ~~ \mathrm{where}~~\psi_0,\psi_1\in\Pi^1_1.\]
It can be easily verified that $\Sigma^1_1\bou\Acz \vdash\Pi^1_1\bou\mathsf{SEP}_0\vdash \Delta^1_1\bou\Caz$.
The converse does not hold.
\begin{theorem}[\cite{Montalban2008-MONOTS} Theorem 3.1,Theorem 2.1]
	The following both hold.
	\begin{enumerate}
		\item There exists an $\omega$-model satisfies $\Delta^1_1\bou\Caz$ and does not satisfies $\Pi^1_1\bou\mathsf{SEP}_0$.
		\item There exists an $\omega$-model satisfies $\Pi^1_1\bou\mathsf{SEP}_0$ and does not satisfies $ \Sigma^1_1\bou\Acz$.
	\end{enumerate}
	The $\omega$-model of 1 will be denoted as $M_m$ hereafter.
\end{theorem}

\begin{theorem}[\cite{Montalbn2006IndecomposableLO} Theorem 5.1]
	There exists an $\omega$-model satisfies $\JI_0$ and does not satisfies $\unique\Pi^1_0\bou\Acz$\footnote{In the literature, an $\omega$-model that fails to satisfy a theory $\mathsf{CDG}\bou\mathsf{CA}$, which is slightly stronger than $\JI$, is constructed.}.
\end{theorem}

In the early literature of reverse mathematics \cite{Friedman1975}, two mathematical theories of hyperarithmetical analysis are mentioned.These are restricted by limitations on arithmetic formulas, so they are not recognized as purely mathematical theorems today.
\begin{definition}[$\mathsf{ABW},\mathsf{SL}$]
When referring to convergence or accumulation points below, we consider the Cantor space as the topology.

The following statement is called arithmetic Bolzano-Weierstrass, denoted as $\mathsf{ABW}$.
\begin{align*}
    A(X) &\text{ has finitely many solutions} \lor \\
    &\text{there exist accumulation points for } \{ X \mid A(X) \}
    \quad \text{where } A(X) \in \Pi^1_0.
\end{align*}
Here, ``$A(X)$ has finitely many solutions'' is an abbreviation for $\exists Z \exists m \forall Y( A(Y) \leftrightarrow \exists n<m  Y=A_n )$.
Furthermore, $Z \subset \N$ is an accumulation point of $\set{X|A(X)}$ if the following holds:
\[\forall n \exists X[A(X) \land X\neq Z \land X[n]= Z[n] ].\]
The following statement is called a sequential limit system, denoted as $\mathsf{SL}$.
\begin{align*}
	Z &\mathrm{~is~an~accumulation~point~of~} \mathcal{A}=\set{X|A(X)} \rightarrow\\
	&\mathrm{there~exists~a~sequence~} \{Z_n\}_{n\in\N}\mathrm{~of~points~of~}\mathcal{A}\mathrm{~that~converges~to}Z, ~~~\mathrm{where}~A(X) \in\Pi^1_0.
\end{align*}

Here, $\{Z_n\}_{n\in\N}$ is called a sequence of points in $A$ if $\forall n A(Z_n)$ holds, and it is said to converge to $Z$ if the following condition is met:
\[\forall m \exists N \forall n>N (Z_n[m] = Z[m]).\]
Each of these, when combined with $\Rcaz$, is denoted by appending a '0' at the bottom.
\end{definition}
\begin{remark}
In the 1970s, when Friedman introduced the two theories mentioned above, he considered including full induction axioms in them.Therefore, in his papers, it is stated without proof that $\mathsf{SL}$ is equivalent to what we now call $\Sigma^1_1\bou\mathsf{AC}$ over $\Rcaz + $ full induction (\cite{Friedman1975} Theorem 2.1\footnote{In Friedman's paper, it is written as HAC instead of $\Sigma^1_1\bou\mathsf{AC}$, which in modern notation would be $\Delta^1_1\bou\mathsf{AC}$.These are equivalent over $\Rcaz$.}), but it is important to note that considerations of limitations on induction are not addressed there.
Chris J.Conidis provided the proofs for the two equivalences that were omitted (\cite{Conidis_2012} Theorem 2.1), where it is stated in the paper that the equivalences can be proven over $\Rcaz$ without using $\Sigma^1_1\bou\Ind$.However, there are explicit uses of $\Sigma^1_1\bou\Ind$ in the proof (p4477, line 14), so at least to the author, it remains unclear whether induction is actually necessary or not.
\end{remark}
Friedman also discussed the results related to $\mathsf{ABW}$ and $\mathsf{SL}$ but did not provide proofs.Conidis gave proofs for each of these and conducted a detailed analysis, especially of $\mathsf{ABW}$.

\begin{theorem}[\cite{Conidis_2012} Theorem 2.1 ]
	The following are true over $\Rcaz+\Sigma^1_1\bou\Ind$:
\begin{itemize}
    \item $\Sigma^1_1\bou\Acz\rightarrow \mathsf{SL}$.
    \item $\Sigma^1_1\bou\Acz\rightarrow \mathsf{ABW}$.
    \item $\mathsf{SL}\rightarrow \Sigma^1_1\bou\Acz$.
    \item $\mathsf{ABW}\rightarrow \unique\Pi^1_0\bou\Acz$.\footnote{Friedman did not mention this fourth item.}
\end{itemize}

\end{theorem}

Forty years after the introduction of $\mathsf{ABW}$ and $\mathsf{SL}$, Montalb{\'{a}}n finally discovered a theorem in pure mathematics concerning ordered sets, known as $\mathsf{INDEC}$, which belongs to hyperarithmetic analysis.
 The theorem that led to the development of $\mathsf{INDEC}$ is said to originate from Pierre Jullien's 1969 doctoral paper.
	

\begin{definition}[\cite{Montalbn2006IndecomposableLO} Statement 1.3]
	We denote the following statement as $\mathsf{INDEC}$:

	Any linear order that is dispersed and indecomposable is either indecomposable to the right or indecomposable to the left.

	Let us explain the terms used.Assume a linear order $\mathcal{A}=\braket{A,\le}$ is given.A cut in $\mathcal{A}$ is a pair of subsets $\braket{L,R}$ of $A$ such that $L= A\backslash R$, where $L$ is an initial segment of $\mathcal{A}$.
	\begin{itemize}
			\item $\mathcal{A}$ is said to be indecomposable if, for any cut $\braket{L,R}$, $\mathcal{A}$ can be embedded into either $L$ or $R$ (the orders on $L$ and $R$ are induced by the order of $\mathcal{A}$).
			\item $\mathcal{A}$ is indecomposable to the right if, for any cut $\braket{L,R}$ where $R\neq \varnothing$, $\mathcal{A}$ can be embedded into $R$.
			\item Indecomposability to the left is defined similarly to the right.
			\item $\mathcal{A}$ is dispersed if it cannot embed $\Q$, the usual order of rational numbers.
	\end{itemize}
\end{definition}

\begin{theorem}[\cite{Montalbn2006IndecomposableLO} Theorem 2.2]
	$\Delta^1_1\bou\Caz \vdash \mathsf{INDEC}$.
\end{theorem}
Itay Neeman conducted a detailed analysis of this $\mathsf{INDEC}$.
\begin{theorem}[\cite{Neeman2008TheSO}]
	$\Rcaz + \Sigma^1_1\bou\Ind \vdash \Delta^1_1\bou\Caz \rightarrow \mathsf{INDEC} \rightarrow \unique\Pi^1_0\bou\Acz$, and furthermore, the converses do not hold and are separated by $\omega$-models.
\end{theorem}
Furthermore, Neeman showed that $\Sigma^1_1\bou\Ind$ is essential in deriving $\unique\Pi^1_0\bou\Acz$ from $\mathsf{INDEC}$ (\cite{Neeman2009NECESSARYUO} Theorem 1.1).

Given the results discussed above and those that will be presented by Conidis below, it is evident that $\Delta^1_1\bou\Caz$ and $\mathsf{ABW}$ are incomparable over $\Rcaz + \Sigma^1_1\bou\Ind$.
\begin{theorem}[\cite{Conidis_2012} Theorem 3.1, Theorem 4.1]
The following holds:
\begin{itemize}
    \item There exists an $\omega$-model that satisfies $\Delta^1_1\bou\Caz$ but does not satisfy $\mathsf{ABW}$.
    \item There exists an $\omega$-model that satisfies $\mathsf{ABW}$ but does not satisfy $\mathsf{INDEC}$.
\end{itemize}

\end{theorem}
\begin{theorem}[\cite{Conidis_2012} Theorem4.7]\label{Theorem:ConidisがVanWesepのモデルはABWも充足してる定理}
	The $\omega$-model $M_w$, constructed by Van Wesep in the proof of \cite{VanWesep1977} Theorem 1.1, additionally satisfies $\mathsf{ABW}$.
\end{theorem}





\section{Unique version of the Dependent Choice Axiom
}\label{subsection:従属選択公理の一意版}

In this section, we will discuss the following theories.
\begin{definition}\label{definition:選択公理各種の一意版}
	The names on the left are given to the axiom schemes consisting of the universal closures of the formulas on the right.
\begin{align*}
	\unique\Gamma\bou\mathrm{DC}~:~&\forall X \exists ! Y \varphi(X,Y) \rightarrow  \exists Y  \forall n \varphi(Y^n ,Y_{n}) ~~~ \mathrm{where}~\varphi\in \Gamma \\
	\unique\Gamma\bou\mathrm{TDC}~:~&\forall X \exists ! Y \varphi(X,Y)\land WO(Z) \rightarrow  \exists Y \forall a\in\field(Z) \varphi(Y^a ,Y_{a}) ~~~ \mathrm{where}~\varphi\in \Gamma
\end{align*}
And thus the following $\Lt$ theory is defined.
\begin{align*}
	\unique\Gamma\bou\Dcz &:= \unique\Gamma\bou\mathrm{DC} + \Rcaz \\
	\unique\Gamma\bou\Tdcz &:= \unique\Gamma\bou\mathrm{TDC} + \Rcaz
\end{align*}
\end{definition}
$\unique\Gamma\bou\mathrm{DC}$ is likely first introduced in this paper, whereas $\unique\Sigma^1_1\bou\mathrm{TDC}$ was introduced in \cite{Ruede02}.There, it is denoted as $\weak\Sigma^1_1\bou\mathrm{TDC}$, and primarily, the strength of the version without the uniqueness condition, $\Sigma^1_1\bou\mathrm{TDC}$, has been examined.

R{\"{u}}ede did not prove it, but $\Atrz$ and $\unique \Sigma^1_1\bou\Tdcz$ are equivalent \footnote{For instance, it is explicitly stated in \cite{Michael2021} Corollary 2.12.}.Let us briefly prove it.$\unique \Sigma^1_1\bou\Tdcz$ is clearly equivalent to $\Delta^1_1\bou\mathsf{TR}_0$.Furthermore, according to \cite{simpson_2009} Theorem V.5.1, $\Atrz$ and $\Sigma^1_1\bou\mathsf{SEP}_0$ are equivalent, and from the proof therein, the equivalence of $\Sigma^1_1\bou\mathsf{SEP}_0$ and $\Delta^1_1\bou\Trz$ immediately follows.Thus, we have $\Atrz \equiv \Sigma^1_1\bou\mathsf{SEP}_0 \equiv \Delta^1_1\bou\Trz \equiv \unique \Sigma^1_1\bou\Tdcz$.Let us summarize this as a proposition.

\begin{proposition}
 $\Atrz \equiv \Sigma^1_1\bou\mathsf{SEP}_0 \equiv \Delta^1_1\bou\Trz \equiv \unique \Sigma^1_1\bou\Tdcz$.
\end{proposition}

$\unique\Sigma^1_1\bou\mathrm{DC}$, by its form, obviously follows from $\unique \Sigma^1_1\bou\Tdcz$, that is, $\Atrz$.For the same reason, it also follows from $\Sigma^1_1\bou\Dcz$.Moreover, as is well-known, $\Sigma^1_1\bou\Dcz$ and $\Atrz$ are incomparable.Therefore, $\unique\Sigma^1_1\bou\mathrm{DC}$ is strictly weaker than both.
On the other hand, since $\unique\Pi^1_0\bou\mathrm{DC}$ allows for the iteration of $\omega$ Turing jumps, it cannot be proved from $\Sigma^1_1\bou\Acz$, which is conservative over $\Acaz$ for $\Pi^1_2$ (\cite{simpson_2009}, Theorem IX.4.4).Thus, the following statement holds.

\begin{proposition}\label{proposition:Thm(ATR)∩Thm(Σ11DC)からuniquΣ11DCだせる}
The intersection of theorems, $\mathrm{Thm}(\Atrz) \cap \mathrm{Thm}(\Sigma^1_1\bou\Dcz)$, is stronger than $\unique \Sigma^1_1\bou\Dcz$, which is at least as strong as $\unique \Pi^1_0\bou\Dcz$, which in turn is at least as strong as $\Acaz^+$.Additionally, $\Sigma^1_1\bou\Acz \not\vdash \unique \Pi^1_0\bou\Dcz$.
Here, $\mathrm{Thm}(T) = \set{ \sigma | T\vdash \sigma}$ denotes the set of theorems provable by theory $T$.
\end{proposition}

We will now discuss the relationship between $\unique \Gamma \bou\Dcz$ and $\unique \Gamma \bou\Acz$ and proceed to strengthen this proposition.

Firstly, similar to the standard $\Gamma \bou \mathrm{DC}$, $\unique\Gamma\bou\mathrm{DC}$ can also be reformulated as follows:
\begin{proposition}\label{proposition:DCの言い換え}
	When $\Gamma$ is either $\Sigma^1_1$ or $\Pi^1_0$, the following diagram is equivalent over $\Rcaz$:
\begin{enumerate}
	\item $\unique\Gamma\bou\mathrm{DC}$.
	\item $\forall X \exists ! Y \theta(X,Y) \rightarrow \forall A \exists Y (Y_0= A \land \forall n \theta(Y_n ,Y_{n+1})) ~~~ \mathrm{where}~\theta\in \Gamma$.
	\item $\forall n \forall X \exists ! Y \psi(n,X,Y) \rightarrow \forall A \exists Y (Y_0= A \land \forall n \psi(n,Y_n ,Y_{n+1})) ~~~ \mathrm{where}~\psi\in \Gamma$.
\end{enumerate}
\end{proposition}
\begin{proof}
Although straightforward, the proof is provided to ensure that the reader can visually verify that the conditions on the class of formulas are safe.Below, it is assumed that $X = X_L \oplus X_R$.

$(1\rightarrow 3)$ 
Assume $\forall n\forall X \exists ! Y \psi(n,X,Y)$, and let $A$ be arbitrary.Here, let $\varphi(X,Y)$ be the following formula:
\begin{align*}
[X = \varnothing &\rightarrow Y= \{0\}\oplus A]\land \\
[\exists n ((X_{n})_L = \{n\} \land \forall m>n ( X_m = \varnothing) )  &\rightarrow  \exists n ((X_{n})_L = \{n\} \land \psi(n,(X_n)_R,Y_R) \land Y_L=\{n+1\}] \land \\
[\mathrm{otherwise} &\rightarrow Y=\varnothing).]
\end{align*}
From the assumption, it follows that $\forall X \exists ! Y \varphi(X,Y)$ holds.
Given this, if one takes $\braket{Y_n}_{n\in\N}$ such that $\forall n\psi(Y^n,Y_{n})$, then defining $\braket{Z_n}_{n\in\N}$ where $Z_n := (Y_{n})_R$, ensures that $Z_0 = A_0$ and $\forall n \psi(n,Z_n,Z_{n+1})$ are satisfied.

$(3\rightarrow 1)$
Assume $\forall X \exists ! Y \varphi(X,Y)$.Now, let $\psi(n, X, Y)$ be the following formula:
\[\forall i< n(Y_ i = X_i )\land \varphi(X,Y_n) \land \forall i>n ( Y_i = \varnothing).\]
From the assumption, it follows that $\forall n\forall X \exists ! Y \psi(n,X,Y)$.By using the assumption 3, if one takes $\braket{Y_n}_{n\in\N}$ such that $Y_0 = \varnothing$ and $\psi(n,Y_n,Y_{n+1})$, then defining $\braket{Z_n}_{n\in\N}$ where $Z_n := (Y_{n+1})_n$ ensures that $\forall n \varphi(Z^n,Z_{n})$ is satisfied.

$(2\rightarrow 3)$ 
Assume $\forall n \forall X \exists ! Y \psi(n,X,Y)$ and let $A$ be arbitrary.Now, let $\theta(X,Y)$ be the following formula:
\begin{align*}
	[\exists n(X_L = \{n\}) &\rightarrow \exists n(X_L = \{n\} \land Y_L=\{n+1\} \land \psi(n,X_R,Y_R)
	)]\land \\
	[\lnot \exists n(X_L = \{n\}) &\rightarrow Y=\varnothing .]
\end{align*}
From the assumption, it follows that $\forall X \exists ! Y \theta(X,Y)$.According to the assumption 2, if one takes $\braket{Y_n}_{n\in\N}$ such that $Y_0 = \{0\} \oplus A$ and $\theta(Y_n, Y_{n+1})$, then defining $\braket{Z_n}_{n\in\N}$ where $Z_n := (Y_n)_R$ ensures that $Z_0 = A$ and $\forall n \psi(n, Z_n, Z_{n+1})$ are satisfied.

$(3 \rightarrow 2)$ is obvious.
\end{proof}
In particular, from the third reformulation, it is clear that when $\Gamma$ is either $\Pi^1_0$ or $\Sigma^1_1$, $\unique\Gamma\bou\Dcz \vdash \unique\Gamma\bou\Acz$ follows.Together with the fact mentioned in Proposition \ref{proposition:Thm(ATR)∩Thm(Σ11DC)からuniquΣ11DCだせる}, it is evident that both $\unique\Pi^1_0\bou\Dcz$ and $\unique\Sigma^1_1\bou\Dcz$ belong to hyperarithmetic analysis.Although the converse of $\unique\Gamma\bou\Dcz \vdash \unique\Gamma\bou\Acz$ does not hold, it can be derived by adding appropriate induction.

As with $\unique \Gamma \bou \Acz$, $\unique \Gamma \bou \Dcz$ is equivalent from $\Pi^0_2$ to $\Pi^1_0$ (cf.~\cite{SuzukiweakAC}, Remark 7).

\begin{lemma}\label{lemma:uniqueArithAC+Σ11INDproveuniqueArithDC}
The following is true for every $k \in \omega$:
\[\unique\Pi^1_0\bou\Acz +\Sigma^1_1\bou\Ind \vdash\unique\Pi^1_0\bou\Dcz,\]	
\[\unique\Sigma^1_{k+1}\bou\Acz +\Sigma^1_{k+1}\bou\Ind \vdash\unique\Sigma^1_{k+1}\bou\Dcz.\]	
\end{lemma}
\begin{proof}
Since the case is similar for other conditions, we will prove assuming $\theta \in \Sigma^1_1$.Assume $\forall n X \exists ! Y \theta(n, X, Y)$ and let $A$ be arbitrary.Define the following formula as $\psi(n, Z)$:
\[
Z_0 = A \land \forall i < n \theta(i, Z_i, Z_{i+1}) \land \forall i > n (Z_i = \varnothing).
\]
From Lemma \ref{lemma:Σ1k上Σ1kは有界任意量化に閉じる}, $\psi$ is also in $\Sigma^1_1$.Therefore, by $\Sigma^1_1$ induction, $\forall n \exists Z \psi(n, Z)$ can be shown.Uniqueness is obvious, thus $\forall n \exists !Z \psi(n, Z)$ holds.Consequently, a solution can be constructed via $\unique\Sigma^1_1\bou\Acz$.
\end{proof}
Therefore, the $\omega$-models of $\unique\Gamma\bou\Acz$ and $\unique\Gamma\bou\Dcz$ completely coincide.Particularly from this fact, the following holds with the $\omega$-model $M_w$(cf.Theorem \ref{Theorem:ConidisがVanWesepのモデルはABWも充足してる定理}) as witness:
\begin{corollary}
	$\unique\Pi^1_0\bou\Dcz < \unique\Sigma^1_1\bou\Dcz.$
\end{corollary}
Thus, we have confirmed everything up to what is written in Figure \ref{figure:Unique choice part of the Reverse Mathematics Zoo: Hyperarithmetic Analysis Area}.

\begin{figure}[hbtp]
	\centering
	\includegraphics[width=12cm,pagebox=cropbox,clip]{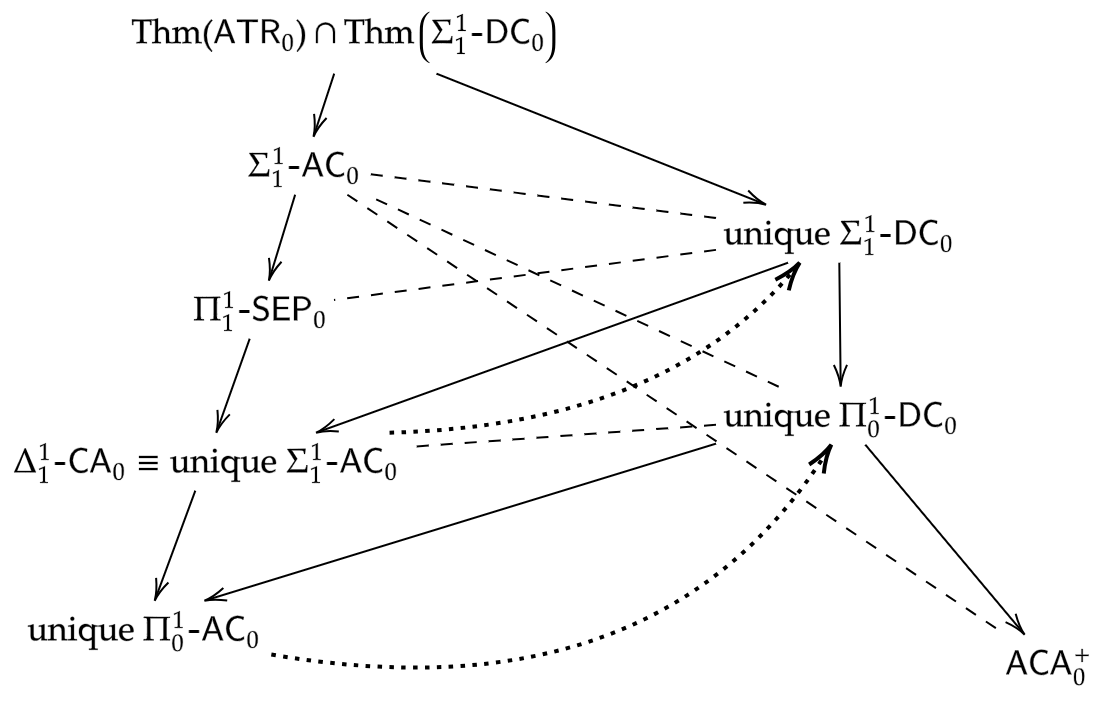}
	\caption{
		Unique choice part of the Reverse Mathematics Zoo: Hyperarithmetic Analysis Area
		}\label{figure:Unique choice part of the Reverse Mathematics Zoo: Hyperarithmetic Analysis Area}
\end{figure}

\newpage

\section{Finite version of the Dependent Choice Axiom}\label{subsection:Finite version of the Dependent Choice Axiom}

Regarding the axioms listed in Definition \ref{definition:選択公理各種の一意版}, it is natural to consider axioms where the phrase ``exactly one'' in the antecedent is replaced with ``a finite number''.

Define ``$\exists \mathrm{~nonzero~finitely~many~} X \psi(X)$'' as an abbreviation for
\[\exists m\exists X \forall Y( \psi(Y) \leftrightarrow \exists n\le m( Y=X_n)),\]
and following Definition \ref{definition:選択公理各種の一意版}, define terms like $\finite\Gamma\bou\mathrm{DC}$.Proposition \ref{proposition:DCの言い換え} also holds for this finite version.

$\finite\Pi^1_0\bou\Acz$ was introduced and its properties investigated in \cite{goh_2023}\footnote{Note that in \cite{goh_2023}, $\finite\Pi^1_0\bou\Acz$ is referred to as $\finite\Sigma^1_1\bou\Acz$.}.In this paper, we use the $\omega$-model constructed by Goh to separate several theories, beginning with an introduction to Goh's results that will be necessary for the subsequent discussions.
\begin{theorem}[\cite{goh_2023} Theorem1.5]
	$	\Rcaz + \Sigma^1_1\bou\Ind + \mathsf{ABW} \vdash \finite\Pi^1_0\bou\Acz$.
	\end{theorem}
As Goh also mentions immediately following the aforementioned theorem, the following corollary holds from Theorem \ref{Theorem:ConidisがVanWesepのモデルはABWも充足してる定理}.
\begin{corollary}
    There exists an $\omega$-model that satisfies $\finite\Pi^1_0\bou\Acz$ and does not satisfy $\Delta^1_1\bou\Caz$.
\end{corollary}

\begin{theorem}[\cite{goh_2023} Theorem 1.3]\label{theorem:M_g}
  There exists an $\omega$-model that satisfies $\Delta^1_1\bou\Caz$ but does not satisfy $\finite\Pi^1_0\bou\Acz$.This $\omega$-model is denoted as $M_g$.
\end{theorem}
From Goh's results, it is particularly evident that $\finite\Pi^1_0\bou\Acz$ and $\Delta^1_1\bou\Caz$ are incomparable.



Similarly to the unique versions, $\finite\Gamma\bou\Dcz$ implies $\finite\Gamma\bou\Acz$, and the converse is also true under appropriate induction.
By rephrasing ``finite'', the complexity of the necessary induction axiom can be lowered.
Now, we use ``$\exists $ nonzero \textit{subfinitely} many $X \varphi(X)$'' as an abbreviation for the following weakened statement:
\[\exists m\exists X \forall Y( \varphi(Y) \rightarrow \exists n\le m( Y=X_n)).\]
\begin{proposition}\label{proposition:subfiniteとfintieの帰納法上での同値性}
	When $\theta$ is a $\Pi^1_0$ formula and $\varphi$ is a $\Sigma^1_{k+1}$ formula, the following is true:
\[
\Rcaz + \Pi^1_{0}\bou\Ind \vdash \exists \nfm X \theta(X) \leftrightarrow \exists \nsfm X \theta(X)
\]
\[
\Rcaz + \Sigma^1_{k+2}\bou\Ind \vdash \exists \nfm X \varphi(X) \leftrightarrow \exists \nsfm X \varphi(X)
\]
In particular, on $\Rcaz + \Sigma^1_{k+2}\bou\Ind$, ``$\exists \nfm X \varphi(X)$'' is $\Sigma^1_{k+2}$.
	\end{proposition}
	\begin{proof}
	Assume $m$ and $X=\braket{X_n}_{n\le m}$ satisfy $\forall Y( \varphi(Y)\rightarrow \exists n\le m (Y=X_n))$.From Lemma \ref{lemma:Σ1k上Σ1kは有界任意量化に閉じる}, the following can be shown through induction over $n$:
	\[
	\forall n\le m \exists \sigma\in2^n\forall i\le n (\sigma(i)=1 \leftrightarrow \varphi(X_i) ).
	\]
	Using the $\sigma\in2^m$ obtained here, it is possible to remove elements from $\braket{X_n}_{n\le m}$ that do not constitute solutions to $\varphi$.	
	\end{proof}
\begin{theorem}\label{theorem:finiteΣ1kAC+Σ1k+2INDprovefiniteΣ1kDC}
		The following is true for every $k \in \omega$:
		\[\finite\Pi^1_0\bou\Acz + \Sigma^1_{2}\bou\Ind \vdash \finite\Pi^1_0\bou\Dcz ,\]
		\[\finite\Sigma^1_{k+1}\bou\Acz + \Sigma^1_{k+2}\bou\Ind \vdash \finite\Sigma^1_{k+1}\bou\Dcz .\]
\end{theorem}
\begin{proof}
Assume $\varphi \in \Sigma^1_{k+1}$ and that $\forall n \forall X \exists \mathrm{~nonzero~finitely~many~} Y \varphi(n, X, Y)$, and let $A$ be arbitrary.Define the following formula as $\psi(n, Z)$:
\[Z_0 = A \land \forall i < n \varphi(i, Z_i, Z_{i+1}) \land \forall i > n (Z_i = \varnothing).\]
From Lemma \ref{lemma:Σ1k上Σ1kは有界任意量化に閉じる}, $\psi$ is in $\Sigma^1_{k+1}$.Therefore, using $\Sigma^1_{k+2}\bou\Ind$ and the resulting $\bounded\Pi^1_{k+1}\bou\Acz$ from Proposition \ref{proposition:subfiniteとfintieの帰納法上での同値性}, it can be verified that $\forall n \exists \mathrm{~nonzero~finitely~many~} Z \psi(n, Z)$.

From the above, we can take $\braket{Z_n}_{n\in\N}$ such that $\forall n \psi(n,Z_n)$ is satisfied.Although each $Z_n$ is a finite sequence of sets of length $n+1$, it is not necessarily true that $Z_n \subset Z_{n+1}$.To construct the desired infinite sequence from such $\braket{Z_n}_{n\in\N}$, we define an indexing tree $T \subset \N^{<\N}$, which can be used to construct an infinite sequence from an infinite path.
  
Here is how we recursively define the sequence of finite sequences $\braket{\sigma_n}_{n\in\N}$ and the sequence of finite trees $\braket{T_n}_{n\in\N}$:
\begin{itemize}
  \item Define $\sigma_0 = \braket{}$, $\sigma_1 = \braket{1}$,
  \item Define $T_0 = \braket{}$, and $T_1 = T_0 \cup \set{ \tau | \tau \subset \sigma_1} = \{\sigma_1\}$.
\end{itemize}
Assume that $\sigma_n$ and $T_n$ have been defined up to $n$.

	In the notation used within this proof, when we define $ Z_n = \braket{A_0^n, A_1^n, A_2^n, \dots, A_{n-1}^n, A_n^n},$
	the sequence truncated to the first $k+1$ elements, $\braket{A_0^n, A_1^n, A_2^n, \dots, A_{k-1}^n, A_k^n}$, is denoted by $Z_n\restriction k$.
  
	Define $\sigma_{n+1}$ using the following algorithm, and define $T_{n+1} = T_n \cup \set{ \tau | \tau \subset \sigma_{n+1}}$.
	\begin{enumerate}
    \item[$0.$] Search for a path of length $n$ in $T_n$: 
    \begin{itemize}
        \item If $Z_{n+1}\restriction n = Z_n$, then set $\sigma_{n+1} = {\sigma_n}^\frown \braket{n+1}$.
        \item Otherwise, continue below.
    \end{itemize}

    \item[$1.$] Search for a path of length $n-1$ in $T_n$:
    \begin{itemize}
        \item[$1\bou 0.$] If $Z_{n+1}\restriction n-1 = Z_n \restriction n-1$, then set $\sigma_{n+1} = {\sigma_n[n-1]}^\frown \braket{n+1, n+1}$.
        \item[$1\bou 1.$] If $Z_{n+1} \restriction n-1 = Z_{n-1}$, then set $\sigma_{n+1} = {\sigma_{n-1}}^\frown \braket{n+1, n+1}$.
        \item Otherwise, continue below.
    \end{itemize}

    \item[$2.$] Search for a path of length $n-2$ in $T_n$:
    \begin{itemize}
        \item[$2\bou 0.$] If $Z_{n+1}\restriction n-2 = Z_n \restriction n-2$, then set $\sigma_{n+1} = {\sigma_n[n-2]}^\frown \braket{n+1, n+1, n+1}$.
        \item[$2\bou 1.$] If $Z_{n+1} \restriction n-2 = Z_{n-1} \restriction n-2$, then set $\sigma_{n+1} = {\sigma_{n-1}[n-2]}^\frown \braket{n+1, n+1, n+1}$.
        \item[$2\bou 2.$] If $Z_{n+1} \restriction n-2 = Z_{n-2}$, then set $\sigma_{n+1} = {\sigma_{n-2}}^\frown \braket{n+1, n+1, n+1}$.
        \item Otherwise, continue below.
    \end{itemize}

    \item[$3.$] Search for a path of length $n-3$ in $T_n$:
    \[ \vdots \]

    \item[$n-1.$] Search for a path of length $1$ in $T_n$:
    \begin{itemize}
        \item[$n-1 \bou 0.$] If $Z_{n+1}\restriction 1 = Z_n \restriction 1$, then set $\sigma_{n+1} = {\sigma_n[1]}^\frown \braket{n+1, ..., n+1}$.
        \item[$n-1 \bou 1.$] If $Z_{n+1} \restriction 1 = Z_{n-1} \restriction 1$, then set $\sigma_{n+1} = {\sigma_{n-1}[1]}^\frown \braket{n+1, ..., n+1}$.
        \item[$n-1 \bou 2.$] If $Z_{n+1} \restriction 1 = Z_{n-2} \restriction 1$, then set $\sigma_{n+1} = {\sigma_{n-2}[1]}^\frown \braket{n+1, ..., n+1}$.
        \[ \vdots \]
        \item[$n-1 \bou n-1.$] If $Z_{n+1} \restriction 1 = Z_1$, then set $\sigma_{n+1} = {\sigma_1}^\frown \braket{n+1, ..., n+1}$.
        \item Otherwise, continue below.
    \end{itemize}

    \item[$n.$] Case where no prefix of paths in $T_n$ matches:
    \[ \sigma_{n+1}  = \underbrace{\braket{1, ..., n+1}}_{n+1 items} \]
\end{enumerate}

	It is clear that both $\braket{\sigma_n}_{n\in\N}$ and $\braket{T_n}_{n\in\N}$ can be defined arithmetically.
Define $T := \bigcup_{n \in \N} T_n$.
Since $\forall n \forall X \exists \mathrm{~nonzero~finitely~many~} Y \varphi(n, X, Y)$, $T$ is a finitely branching infinite tree.
Therefore, by König's lemma, an infinite path $f$ can be obtained.

	Finally, define the sequence $\braket{W_n}_{n\in\N}$ where $W_n = (Z_{f(n)})_n = A^{f(n)}_n$.This sequence satisfies $\forall n \varphi(W_n, W_{n+1})$.

\end{proof}


	In summary, the satisfaction relationships of each hyperarithmetic analysis theory in the models $M_w$(cf.Theorem \ref{Theorem:ConidisがVanWesepのモデルはABWも充足してる定理}) and $M_g$(cf.Theorem \ref{theorem:M_g}) are organized as shown in the following table.
	\begin{table}[h]
    \caption{
			Satisfaction relations for some hyperarithmetic analysis theories with models $M_w$ and $M_g$.
			}
    \label{table:MwMgtable}
		\begin{center}
	\begin{tabular}{|c|c|c|} \hline
		 & $M_w  $ &  $M_g $  \\ \hline
		$\unique \Pi^1_0\bou\Acz $  & Yes & Yes\\
		$\unique \Pi^1_0\bou\Dcz $  & Yes &  Yes \\\hline
		$\finite \Pi^1_0\bou\Acz $ &  Yes & No\\
		$\finite \Pi^1_0\bou\Dcz $ &  Yes & No\\ \hline
		$\Delta^1_1\bou\Caz $ & No & Yes\\ 
		$\unique \Sigma^1_1\bou\Dcz $  &No  & Yes\\ \hline
		$\finite \Sigma^1_1\bou\Acz $ & No & No\\
		$\finite \Sigma^1_1\bou\Dcz $ & No & No\\ \hline
		$\Sigma^1_1\bou\Acz $ & No & No \\ \hline
		\end{tabular}
		\end{center}
	\end{table}
	In the table, ``Yes'' indicates that the $\omega$-model satisfies the axiom listed to the left, while ``No'' indicates the opposite.By synthesizing the information summarized in this table and the fact that the $\mathsf{DC}$ series implies $\Acaz^+$, whereas the $\mathsf{AC}$ series does not, we can draw the diagram introduced in the introduction as Figure \ref{figure:Reverse Mathematics Zoo: Hyperarithmetic Analysis Area}.
\newpage

\section{\texorpdfstring{Approximation of Hyperarithmetic Analysis by the $\omega$-model reflection}{Approximation of Hyperarithmetic Analysis by the omega-model reflection}}
In this section, we move beyond examining relationships within individual theories as discussed in sections \ref{subsection:従属選択公理の一意版} and \ref{subsection:Finite version of the Dependent Choice Axiom}, to delve into the structural properties of hyperarithmetic analysis.The entirety of these theories is represented as $\mathcal{HA}$.
Furthermore, we establish a preorder within this set, defined by implication relations, and represent it as $(\mathcal{HA}, \le)$.

Now, let us introduce three well-known results about the structure of $(\mathcal{HA}, \le)$.
\begin{theorem}[\cite{VanWesep1977} 2.2.2]
For any $T \in \mathcal{HA}$, there exists a $T$-computable $T' \in \mathcal{HA}$ that is strictly weaker than $T$ and has strictly more $\omega$-models.
\end{theorem}
This theorem particularly reveals that $(\mathcal{HA}, \le)$ contains infinite descending chains.
\begin{proposition}
	$(\mathcal{HA}, \le)$ has a trivial greatest element $\bigcap_{X \subset \omega} \mathrm{Th}(\HYP(X))$, where $\mathrm{Th}(\HYP(X))$ represents all $\Lt$ sentences that are true in the $\omega$-model $\HYP(X)$.
\end{proposition}

\begin{definition}\label{definition:Γ-RFN}
Assume $\varphi$ is a formula with no free variables other than $X_1, ..., X_n$.We denote the following as $\Rfn_\varphi$:
\begin{align*}
    &\forall X_1, ..., X_n (\varphi(X_1, ..., X_n) \\
    &\rightarrow \exists ~\codedom M~s.t.~ (X_1, ..., X_n \in M \land M \models \varphi(X_1, ..., X_n) + \Acaz)).
\end{align*}
Then, define $\Gamma\bou\Rfn := \set{\Rfn_\varphi | \varphi \in \Gamma}$, and as an example, set $\Gamma\bou\Rfn_0 := \Gamma\bou\Rfn + \Acaz$.
\end{definition}
\begin{lemma}\label{lemma:Theorem VIII.5.12の一部}
	$\Sigma^1_1\bou\Dcz \equiv \Sigma^1_3\bou\Rfnz$.
\end{lemma}
\begin{proof}
\cite{simpson_2009} Theorem VIII.5.12.
\end{proof}

\begin{theorem}\label{Thorem:Sigma^1_3文で有限公理化される理論は超算術的解析でない．}
	There are no theories within $\mathcal{HA}$ that can be finitely axiomatized by $\Sigma^1_3$ sentences.
\end{theorem}
\begin{proof}
From the previous lemma, $\HYP = \HYP(\varnothing) \models \Sigma^1_3\bou\Rfnz$.Therefore, for $\Sigma^1_3$ sentences that are true in $\HYP$, there exists a model within $\HYP$, that is, a model strictly smaller than $\HYP$, for these sentences.
\end{proof}
Next, in order to analyze $\mathcal{HA}$, we introduce a class of theories that approximate it.



\begin{definition}
Let $T$ be a recursively axiomatizable $\Lt$ theory.We denote the following statement as $\Rfn(T)$, and call it the $\omega$-model reflection of $T$:
\[ 
\forall X \exists M (X \in M \land M \text{ is a coded}~\omega \text{-model of } T + \Acaz).
\]
Furthermore, for each $n \in \omega$, we recursively define the $n$-fold iteration of the $\omega$-model reflection, $\Rfn^n(T)$, as follows:
\begin{align*}
	\Rfn^0(T) &= T, \\
	\Rfn^{n+1}(T) &= \Rfn( \Rfn^n(T)+\Acaz).
\end{align*}
As an example, we define $\Rfn^n(T)_0 := \Rfn^n(T) + \Acaz$, and specifically denote $\Rfn^1(T)_0 = \Rfn(T + \Acaz) + \Acaz$ as $\Rfn(T)_0$.
\end{definition}

From the strong soundness theorem (\cite{simpson_2009} Theorem II.8.10), the following fundamental fact follows:
\begin{lemma}\label{lemma:RFN(T)→Con(T)}
    For any recursively axiomatizable $\Lt$ theory $T$, $\Acaz \vdash \Rfn(T) \rightarrow \Con(T)$.
\end{lemma}
\begin{proof}
	The case when $T$ is a sentence is straightforward.For $T$ being recursively axiomatizable, especially when it is an infinite set, $\Acaz + \Rfn(T) \vdash \Acaz^+$.This ensures that within the system, the satisfaction relation for a $\codedom$ can be defined, which determines the truth values of all formulas.
\end{proof}

The class of theories defined by this axiom, $\Rfn^{-1}(\Atrz) := \set{ T | \Rfn(T) \equiv \Atrz}$, approximates $\mathcal{HA}$ to some extent.Here, we present two pieces of evidence: similarity in instances and similarity in closure properties.

As previously pointed out, explicitly axiomatized existing theories of hyperarithmetic analysis are encompassed between $\JI_0$ and $\Sigma^1_1\bou\Dcz$.Conversely, the next proposition states that such theories also belong to $\Rfn^{-1}(\Atrz)$.
\begin{proposition}
	If $T$ lies between $\mathsf{JI}_0$ and $\Sigma^1_1\bou\Dcz$, then $T \in \Rfn^{-1}(\Atrz)$.
\end{proposition}
\begin{proof}
According to \cite{simpson_2009} Lemma VIII.4.19, $\Atrz \vdash \Rfn(\Sigma^1_1\bou\mathsf{DC})_0$.
Let's show that $\Rfn(\mathsf{JI})_0 \vdash \Atrz$.
It is enough to choose $a$ and $X$ such that $a \in \ko^X$, and then demonstrate the existence of $\mathrm{H}^X_a$.

From the assumption, there exists a model $M$ such that $X \in M \models \JI$, and moreover, $M \models a \in \ko^X$, specifically implying that:
\[ M \models (\set{b | b <_{\ko}^X a}, <_{\ko}^X) \text{ is a well-order}.\]
Given that $M \models \JI$, it follows that 
$M \models (\forall b <_{\ko}^X a \exists \mathrm{H}^X_b) \rightarrow \exists \mathrm{H}^X_a$.
Therefore, what needs to be shown is
$\forall b <_{\ko}^X a \exists Y \in M (Y = \mathrm{H}^X_b)$.
The formula $\exists Y \in M (Y = \mathrm{H}^X_b)$ is arithmetical with set parameters included in $M$, so we will demonstrate this by arithmetic transfinite induction.If $|b|^X = 0$, it is trivial.If $|b|^X$ is a successor ordinal, it also holds as $M \models \Acaz$.Assume $|b|^X$ is a limit ordinal.By the induction hypothesis, we have $\forall c <_{\ko}^X b \exists Y \in M (Y = \mathrm{H}^X_c)$.
This implies $M \models \forall c <_{\ko}^X b \exists \mathrm{H}^X_c$.
Hence, from $M \models \JI$, it follows that $M \models \exists \mathrm{H}^X_b$, completing the induction.
\end{proof}

Next, we examine the similarity in closure properties.


\begin{lemma}\label{lemmma Ref(T)⇔Ref(T+￢Ref(T))}
	Let $T$ be an $\Lt$ sentence such that $T \vdash \Acaz$.Then, the following holds:
	\[\Acaz \vdash \Rfn(T)\leftrightarrow \Rfn(T+ \lnot \Rfn(T)).\]
\end{lemma}
\begin{proof}
The $(\leftarrow)$ direction is trivial.We now show that $\Acaz + \Rfn(T) \vdash \Rfn(T + \lnot \Rfn(T))$.First, if $\Acaz + \Rfn(T)$ is inconsistent, it is clear, so we may assume it is consistent.
Let us fix a model $M$ and an $X \in M$.What needs to be shown is the existence of an $N \in M$ satisfying:
\[X \in N \models T + \lnot \Rfn(T).\]
The subsequent discussion can be considered a relativization of the $\omega$-model incompleteness discussed in \cite{simpson_2009} Theorem VIII.5.6.

Let $C$ be a set constant.We name the following three $\Lt + \{C\}$ sentences as (1), (2), and (3) respectively:
\begin{enumerate}[(1)]
    \item $T + \lnot \exists N( C \in N \land N \text{ is a coded}~\omega\text{-model of } T)$.
    \item $\exists N( C \in N \land N \text{ is a coded}~\omega\text{-model of } T)$.
    \item $\lnot \exists N( C \in N \land N \text{ is a coded}~\omega\text{-model of } (1))$.
\end{enumerate}
Define $T^\ast := \Acaz + (2) + (3)$.To prove our conclusion, it is sufficient to show that $T^\ast$ is inconsistent.This is because, currently, $\Acaz + (2)$ is consistent (with $(M, C^M := X)$ serving as its model).Therefore, if $T^\ast$ turns out to be inconsistent, it would imply that $\Acaz + (2) \vdash \lnot (3)$.Consequently, $(M, C^M := X) \models \exists N( C \in N \land N \text{ is a coded}~\omega\text{-model of } (1))$.

In this paragraph, we discuss within  $T^\ast (\vdash \Acaz)$.First, from (2), there exists a $\codedom N'$ such that $C \in N' \models T$.If we suppose that $N' \not\models (2)$, then the negation of (3) holds, leading to a contradiction.Therefore, $N' \models (2)$.Since (3) is a valid $\Pi^1_1$ sentence, $N' \models (3)$ as well.Combining these, $N' \models T^\ast$.Given that $T^\ast$ is finitely axiomatized, $\Acaz$ allows us to construct a natural weak model of $T^\ast$ from $N'$.Hence, by the strong soundness theorem (\cite{simpson_2009} Theorem II.8.10), $T^\ast$ is consistent.

Therefore, by the second incompleteness theorem, $T^\ast$ is inconsistent.
\end{proof}

\begin{proposition}
	For an $\Lt$ sentence $T$ such that $T \vdash \Acaz$, the following are all equivalent:
\begin{enumerate}
	\item $T\in \mathcal{HA}$.
	\item $T+ \mathrm{full~induction} \in \mathcal{HA}$.
	\item $T+\lnot \Rfn(T) \in \mathcal{HA}$.
	\item $T+\lnot \ATR\in \mathcal{HA}$.
\end{enumerate}
\end{proposition}
\begin{proof}
	The equivalence $(1 \Leftrightarrow 2)$ is evident because $\omega$-models always satisfy all instances of induction.

$(1 \Rightarrow 3)$  
From (1), for each $X \subset \omega$, $\HYP(X) \models T$.Now, suppose $\HYP(X) \models \Rfn(T)$.This would imply the existence of some $M \in \HYP(X)$ such that $X \in M \models T$.However, this contradicts the fact that $\HYP(X)$ is the smallest $\omega$-model of $T$ containing $X$.Therefore, $\HYP(X) \models \lnot \Rfn(T)$.It follows immediately from 1 that any $\omega$-model of $T + \lnot \Rfn(T)$ is closed under hyperarithmetic reduction.

$(3 \Rightarrow 1)$ 
From 3, it is clear that each $X \subset \omega$, $\HYP(X) \models T$ holds.Next, we demonstrate minimality.Consider an $\omega$-model $M$ such that $X \in M \models T$.If $M \not\models \Rfn(T)$, then it is evident from 3.Assume that $M \models \Rfn(T)$.Then, by Lemma \ref{lemmma Ref(T)⇔Ref(T+￢Ref(T))}, $M \models \Rfn(T + \lnot \Rfn(T))$.Consequently, there exists an $N \in M$ such that $X \in N \models T + \lnot \Rfn(T)$, and from 3, $\HYP(X) \subset N \subset M$.

	$(1 \Rightarrow 4)$ 
 The argument is essentially the same as in $(1 \Rightarrow 3)$.It follows from the fact that for each $X \subset \omega$, $\HYP(X)$ does not model $\ATR$.

	$(4 \Rightarrow 1)$
	The discussion is essentially the same as in $(3 \Rightarrow 1)$.Use the fact that $\omega$-models of $\Atrz$ are closed under hyperarithmetic reduction.

\end{proof}
It is considered that the aforementioned closure property, which arises from the fact that hyperarithmetic analysis is defined using $\omega$-models, also holds for $\Rfn^{-1}(\Atrz)$.

\begin{theorem}
For an $\Lt$ sentence $T$ such that $T \vdash \Acaz$, the following are all equivalent:
\begin{enumerate}
	\item $T\in \Rfn^{-1}(\Atrz)$.
	\item $T+ \mathrm{full~induction}\in \Rfn^{-1}(\Atrz)$.
	\item $T+\lnot \Rfn(T) \in \Rfn^{-1}(\Atrz)$.
	\item $T+\lnot \ATR \in \Rfn^{-1}(\Atrz)$.
\end{enumerate}	
\end{theorem}
\begin{proof}
	$(1\Leftrightarrow 2)$ It is clear.

	$(1 \Leftrightarrow 3)$  
	It follows from Lemma \ref{lemmma Ref(T)⇔Ref(T+￢Ref(T))}.$(1 \Rightarrow 4)$ It follows from Lemma \ref{lemmma Ref(T)⇔Ref(T+￢Ref(T))}.

$(4 \Rightarrow 1)$ First, from 4, $\Atrz \equiv \Rfn(T+\lnot \ATR)_0 \vdash \Rfn(T)_0$.To demonstrate the converse, consider a model $M \models \Rfn(T)_0$ and fix $X \in M$.According to 4, it suffices to find a model of $T+\lnot \ATR$ containing $X$ within $M$.Now, since $M \models \Rfn(T)_0$, there exists an $N\in M$ such that $X \in N \models T$.If $N \not \models \ATR$, then we are done.Assuming $N \models \ATR$, again by 4, $N \models \Rfn(T+\lnot \ATR)_0$, therefore, there exists an $N' \in N$ satisfying $X \in N' \models T+\lnot \ATR$.Since $N' \in M$, this is sufficient.

The remainder follows from Lemma \ref{lemmma Ref(T)⇔Ref(T+￢Ref(T))}.

\end{proof}


As seen above, it can be said that $\Rfn^{-1}(\Atrz)$ approximates $\mathcal{HA}$ to some extent.From this, two types of questions arise that are considered useful for exploring the structure of $\mathcal{HA}$.One is the question of exploring similarities, which investigates what other properties both share.The other is the question of exploring differences, which examines how they differ.

Let us introduce the question concerning similarities.As seen in Theorem \ref{Thorem:Sigma^1_3文で有限公理化される理論は超算術的解析でない．}, there are no $\Sigma^1_3$ sentences in $\mathcal{HA}$.It is a natural question to ask whether this also holds for $\Rfn^{-1}(\Atrz)$.Unfortunately, this question has not been completely resolved, and in this paper, we provide a partial answer in the following sense: namely, that there are no $\Pi^1_2$ finitely axiomatizable theories at least within $\Rfn^{-1}(\Atrz)$.

\begin{theorem}\label{theorem:uniqueSigma10DC0上で任意一意存在型から2回RFNが出る}
Assume that $\theta(X,Y)$ is an arithmetic formula with no free variables other than $X$ and $Y$.Then the following holds:
\[\unique\Pi^1_0\bou\Dcz \vdash \forall X \exists	! Y \theta(X,Y) \rightarrow \Rfn^2(\forall X \exists	! Y \theta(X,Y)).\]	
\end{theorem}
\begin{proof}
We discuss within $\unique\Pi^1_0\bou\Dcz + \forall X \exists ! Y \theta(X,Y)$.Fix $A$.Let $\psi(X,Y)$ be the following arithmetic formula:
\begin{align*}
	\forall m \forall k\le m \forall l,r\le k\{ & (m=0 \rightarrow Y_0 = A) \land \\ 
							   & (m = 4k+1 \rightarrow \theta(X_k,Y_m)) \land \\
							   & (m = 4k+2 \land k=\braket{l,r} \rightarrow [(\Phi^{X_l}_r \in2^\N \rightarrow Y_m = \Phi^{X_l}_r  ) \land \\
								  &~~~~~~~~~~~~~~~~~~~~~~~~~~~~~~~~~~~~(\Phi^{X_l}_r \not\in2^\N \rightarrow Y_m = \varnothing  ) ]) \land \\
							   & (m = 4k+3\land k=\braket{l,r} \rightarrow Y_m = X_l \oplus X_r ) \land \\
							   & (m = 4k+4 \rightarrow Y_m = \TJ(X_k)) 	\}.
\end{align*}
\begin{claim}
$\forall X \exists !Y \psi (X,Y)$.
\end{claim}
\begin{proofofclaim}
Fix $X$.The uniqueness of $Y$ is evident.To show existence, take an arithmetic formula $\rho$ that includes $X$ as a parameter and satisfies $\forall m \rho(m,Y) = \psi (X,Y)$.Given $\forall X \exists ! Y \theta(X,Y)$, it is correct that $\forall m \exists ! Y \rho(m,Y)$.Using $\unique\Pi^1_0\bou\Acz$, we can obtain $Y$ such that $\forall m \rho(m,Y_m)$ is satisfied.
This $Y$ meets the condition.
\end{proofofclaim}
Now consider $X = \{X_n\}_{n\in\N}$ where $X_n = A$ for all $n \in \N$, i.e., a sequence of sets where every element is $A$.From the previous claim and $\unique\Pi^1_0\bou\Dcz$, we obtain $Y$ such that $Y_0 = X$ and $\forall n \psi(Y_n, Y_{n+1})$.Furthermore, again from the previous claim, this $Y$ is also unique.
Summarizing the above, it follows that
\[\forall A \exists ! Y(=\{Y_n\}_{n\in\N})( Y_0 = \{A\}_{n\in\N} \land \forall n \psi(Y_n, Y_{n+1}) )\]
is correct.

Given any arbitrary $A$, from the above we can uniquely determine $Y$, and define $M=\{M_n\}_{n\in\N}$ where $M_{\braket{l,r}} = (Y_l)_r$\footnote{Intuitively, $M = \bigcup_{n\in\N} Y_n$}.This $M$ contains $A$ and is a $\codedom$ that satisfies $\forall X \exists ! Y \theta(X,Y) + \Acaz$.

Now, define $\theta'(A,Y)$ as $Y_0 = \{A\}_{n\in\N} \land \forall n \psi(Y_n,Y_{n+1},A)$.Then, $\forall A \exists ! Y \theta'(A,Y)$ holds, and $\theta'$ is arithmetic.Consequently, by replacing $\theta(X_k,Y_n)$ in the definition of $\psi$ with $\theta'(X_k,Y_n)$ to form $\psi'$, a similar discussion can be conducted.For any given $A$, it is possible to construct a $\codedom N$ that contains $A$ and satisfies $\forall X \exists ! Y \theta'(X,Y)+\Acaz$.Particularly, since $N \models \forall X \exists ! Y \theta'(X,Y)$, it follows that $N \models \Rfn(\forall X \exists ! Y \theta(X,Y))$.Therefore, we can conclude $\Rfn(\Rfn(\forall X \exists ! Y \theta(X,Y))+\Acaz)$, namely $\Rfn^2(\forall X \exists ! Y \theta(X,Y))$.
\end{proof}

\begin{corollary}
		For all arithmetic formulas $\theta(X,Y)$ that have no free variables other than $X$ and $Y$, it holds that $\forall X \exists ! Y \theta(X,Y) \not\in \Rfn^{-1}(\Atrz)$.This means the following is correct:
		\[\Rfn(\forall X \exists ! Y \theta(X,Y))_0 \not\equiv \Atrz .\]
\end{corollary}
\begin{proof}
	Assume $\Atrz \vdash \Rfn (\forall X \exists ! Y \theta(X,Y))_0$.Then $\Atrz \vdash \forall X \exists ! Y \theta(X,Y)$, and from $\Atrz \vdash \unique\Pi^1_0\bou\Dcz$, it follows that $\Atrz\vdash \Rfn^2 (\forall X \exists ! Y \theta(X,Y))$ is correct.Therefore, from Lemma \ref{lemma:RFN(T)→Con(T)}, $\Atrz \vdash \Con(\Rfn (\forall X \exists ! Y \theta(X,Y)) + \Acaz)$, which implies $\Atrz \vdash \Con(\Rfn (\forall X \exists ! Y \theta(X,Y))_0)$.
Thus, if we assume $\Rfn (\forall X \exists ! Y \theta(X,Y))_0 \vdash \Atrz$, then
\[\Rfn (\forall X \exists ! Y \theta(X,Y))_0 \vdash \Con(\Rfn (\forall X \exists ! Y \theta(X,Y))_0) \]
is derived, contradicting the second incompleteness theorem.
\end{proof}
From the proof of Theorem \ref{theorem:uniqueSigma10DC0上で任意一意存在型から2回RFNが出る}, it is clear that the equivalence with the $\omega$-model reflection axiom, similar to what is described in Lemma \ref{lemma:Theorem VIII.5.12の一部}, also holds for $\unique\Pi^1_0\bou\Dcz$.

\begin{corollary}
In Definition \ref{definition:Γ-RFN}, consider $\theta$ as an arithmetic formula and collect only those of the form $\forall X \exists !Y \theta(X,Y,....)$.The theory formed by adding $\Acaz$ to this collection is denoted as $\unique\Pi^1_2\bou\Rfn$.Under this setup, the following holds:

\[\unique\Pi^1_0\bou\Dcz \equiv \unique\Pi^1_2\bou \Rfnz .\]
\end{corollary}
\begin{proof}
By constructing an $\omega$-model as done in the proof of Theorem \ref{theorem:uniqueSigma10DC0上で任意一意存在型から2回RFNが出る}, it can be shown that $\unique\Pi^1_0\bou\Dcz \vdash \unique\Pi^1_2\bou \Rfnz$.

To prove the converse, assume $\forall X \exists ! Y \theta(X,Y)$, and take an arbitrary $A$.By the assumption, there exists a $\codedom M$ such that $A \in M \models \forall X \exists ! Y \theta(X,Y)$.Fix an $m_0$ such that $A = (M)_{m_0}$, and define a function $f: \N \to \N$ as follows:
\[\begin{cases}
    f(0) = m_0 ,\\
    f(n+1) = (\mu m)[\theta(M_{f(n)}, M_m)].
\end{cases}\]
Then, $\forall n \theta(M_{f(n)}, M_{f(n+1)})$ holds.

\end{proof}

Regarding the second type of question, which concerns the differences between hyperarithmetic analysis and $\Rfn^{-1}(\Atrz)$, unfortunately, only trivial facts are known.For instance, $\Sigma^1_1\bou\Acz + \Con(\Atrz)$ belongs to the symmetric difference between the two, but no non-trivial theories that fit this example have yet been found.

\section{Open Problem}
As previously mentioned, the following two issues remain unresolved:
\begin{question}
Are there instances of $\Sigma^1_3$ sentences within $\Rfn^{-1}(\Atrz)$?
\end{question}

\begin{question}
What non-trivial theories exist in the symmetric difference between hyperarithmetic analysis and $\Rfn^{-1}(\Atrz)$?
\end{question}

$\unique\Gamma\bou \Dcz$ can be proven by adding induction axiom to $\unique\Gamma\bou \Acz$.However, for example, since $\unique\Pi^1_0\bou\Dcz$ cannot prove $\Sigma^1_1\bou\Ind$, the induction used there is ``excessive''.This raises interest in identifying the axioms that are precisely necessary and sufficient to prove $\unique\Gamma\bou \Dcz$ from $\unique\Gamma\bou \Acz$.Therefore, the following question arises:
\begin{question}
When $\Gamma$ is either $\Pi^1_0$ or $\Sigma^1_1$, is there an $\Lt$ theory $\alpha$ such that:
\[\unique \Gamma\bou \Acz + \alpha \equiv \unique\Gamma \bou \Dcz.\]
What happens if the term 'unique' is replaced with 'finite'?
\end{question}

 
\bibliographystyle{alpha}
\bibliography{main}

\end{document}